\documentclass[10pt]{amsart}
\usepackage{tikz}
\usepackage{stmaryrd}
\usepackage{amsmath,amssymb}
\usepackage{hyperref, color}
\usepackage{amsfonts}
\usepackage{bm}
\usepackage[portuges,english]{babel}
\usepackage{graphicx}
\usepackage{amscd,color}
\usepackage{amsmath}
\usepackage{amssymb}
\usepackage{empheq,amsmath}
\theoremstyle{plain}

\newtheorem{coro}{\bf Corollary}
\newtheorem{defn}{\bf Definition}

\newtheorem{lemma}{\bf Lemma}
\newtheorem*{lemma*}{\bf Lemma}

\newtheorem{remark}{Remark}

\newtheorem{thm}{\bf Theorem}
\newtheorem*{thm*}{\bf Theorem}
\numberwithin{equation}{section}
\newcommand{\RN}[1]{%
  \textup{\uppercase\expandafter{\romannumeral#1}}%
}

\begin{document}

\title[The Pohozaev-Schoen identity on asymptotically euclidean manifolds]{The Pohozaev-Schoen identity on asymptotically euclidean manifolds:
Conservation identities and their applications}

 \author[R. Avalos and A. Freitas]{R. Avalos$^1$ and A. Freitas$^2$}

\address{$^1$ Departamento de Matem\'atica, Universidade Federal do Cear\'a,\\ R. Humberto Monte, 60455-760, Fortaleza/CE, Brazil.}
\email{rdravalos@gmail.com}

\address{$^2$Departamento de Matem\'{a}tica,
Universidade Fe\-deral da Para\'{i}ba, 58059-900, Jo\~{a}o Pessoa, Para\'{i}ba, Brazil.}
\email{allan@mat.ufpb.br}

\subjclass[2000]{Primary 53C21; Secondary 53C24.}

\keywords{Pohozaev-Schoen Identity; Asymptotically Euclidean manifolds; Generalized solitons; Almost-Schur lemma; Static metrics.}

\begin{abstract}
The aim of this paper is to present a version of the generalized Pohozaev-Schoen identity in the context of asymptotically euclidean manifolds. Since these kind of geometric identities have proven to be a very powerful tool when analysing different geometric problems for compact manifolds, we will present a variety of applications within this new context. Among these applications, we will show some rigidity results for asymptotically euclidean Ricci-solitons and Codazzi-solitons. Also, we will present an almost-Schur-type inequality valid in this non-compact setting which does not need restrictions on the Ricci curvature. Finally, we will show how some rigidity results related with static potentials also follow from these type of conservation principles. 
\end{abstract}

\maketitle
\section{Introduction}\label{s-in}

The classical Pohozaev-Schoen identity \cite{Schoen} lies within a larger set of geometric identities which rely on conservation principles arising from symmetries of specific variational problems, as has been noted and explained in \cite{Gover} (see also \cite{allan} for another variant, in manifolds with boundary, of this useful identity). In this work, the authors encapsulate several geometric identities which have had important impact in geometric analysis within a single identity derived from such symmetry principles. It is important to note that these type of geometric identities have been established for compact manifolds. Taking into account how useful these conservation principles have proved to be when analysing geometric problems, we intend to present a version of the Pohozaev-Schoen identity established in \cite{allan} for asymptotically euclidean (AE) manifolds. These manifolds are complete non-compact manifolds with a particular simple structure at infinity. In general relativity, for example, we often need to consider initial data problems on non-compact manifolds, with natural restrictions on the asymptotic geometry (see, for instance, \cite{CB2},\cite{maxwell} and references therein). The positive mass theorem and the Penrose inequality are example of such situations (see \cite{Bartnik},\cite{Huisken},\cite{Schoen-Yau1},\cite{Schoen-Yau},\cite{Witten}). In fact, geometric problems on this scenario have received plenty of attention since these structures play a central role in general relativity, serving as a model for initial data of isolated gravitational systems. Thus, plenty of analytic tools have been developed in this scenario. In this direction, our aim is to prove the validity of a version of the Pohozaev-Schoen identity in this context, which is sufficiently powerful to be useful when analysing a wide variety of geometric problems on AE manifolds. Explicitly, we will prove the following theorem.\footnote{See Section 2 for the precise definition of the functional spaces involved.}

\begin{thm*}
Let $(M,g)$ be a $H_{s,\delta}$-asymptotically euclidean manifold, with $s>\frac{n}{2}+1$ and $\delta>-\frac{n}{2}$. Also, let $B$ be a symmetric $(0,2)$-tensor field, satisfying $\mathrm{div}_gB=0$, and let $X$ be a vector field on $M$. Suppose $X\in C^1$ is bounded with $\nabla X\in L^2$. Furthermore, suppose that $X\in L^2_{-(\rho+1)}$ and $B\in H_{s+1,\rho}$, with $\rho\geq 0$. Then, the following equality holds:
\begin{align}
\int_M\langle \overset{\circ}{B},\pounds_X g \rangle\mu_g   = \frac{2}{n} \int_M X(\mathrm{tr}_gB)\mu_g + 2 \int_{\partial M}\overset{\circ}{B}(X,\nu)\mu_{\partial M}
\end{align}
\end{thm*}

\bigskip
We will employ this version of the generalized Pohozaev-Schoen identity to analyse different geometric problems, such as generalized solitons on AE manifolds. These soliton-type equations arise as a generalization of Ricci almost solitons, as proposed by Pigola \textit{et al.} \cite{Pigola}. Also these type of equations were analysed in the \textit{extrinsic} context by Al\'ias \textit{et al.} in \cite{ALR}. During the analysis of these generalized solitons, we will prove the following rigidity theorem.

\begin{thm*}
Let $(M^n,g)$ be an  asymptotically euclidean manifold without boundary, $n\geq 3$, where $g$ satisfies the hypotheses of the above theorem, and suppose that the tensor field $B\in H_{s+1,0}$, $s>\frac{n}{2}$, is $g$-divergence-free. Furthermore, let $X\in H_{s+2,-1}$ and $\mathrm{tr}_gB=0$. Then, under these assumptions, any $B$-generalized soliton is $B$-flat.
\end{thm*}

From this theorem, we will extract some interesting geometric consequences. For instance, we will show that any AE Ricci almost soliton with zero scalar curvature is trivial, \textit{i.e}, is isometric to the euclidean flat space. 

\bigskip
Another important result, where we will make use of the generalized Pohozaev-Schoen identity, is the following almost-Schur theorem for AE manifolds (see in Section \ref{applications} a review and motivation for this type of inequalities): 

\begin{thm*}
Let $(M^n,g)$ is an $H_{s+3,\delta}$-asymptotically euclidean manifold without boundary, with $n\geq 3$, $s>\frac{n}{2}$, $\delta>-\frac{n}{2}$ and $\delta\geq -2$. Now, let $B$ be a symmetric $(0,2)$-tensor field such that $\mathrm{div}_gB=0$ and $B\in H_{s+1,\rho}$, with $\rho>-\frac{n}{2} + 2$, furthermore, suppose that $\rho\geq 0$, and denote its trace by  $b\doteq \mathrm{tr}_gB$.  Then:\\
I)  there is a constant $C$ independent of the specific choice of such tensor $B$, such that the following inequality holds
\begin{align}
||b||_{L^2}\leq n\left( \frac{n-1}{n} + C||\mathrm{Ric}_g||_{C^0_2}\right)^{\frac{1}{2}} ||\overset{\circ}{B}||_{L^2},
\end{align}
II) If $\mathrm{Ric}_g$ is non-negative, \textit{i.e}, $\mathrm{Ric}_g(X,X)\geq 0$ $\forall$ $X\in\Gamma(TM)$, then
\begin{align}
||b||_{L^2}\leq \sqrt{n(n-1)} ||\overset{\circ}{B}||_{L^2}.
\end{align}
\end{thm*}

\bigskip
Notice that the hypotheses of the theorem split depending on whether $n=3,4$ or $n\geq 5$. In the first case, the hypotheses on the weights $\delta$ and $\rho$ are $\delta>-\frac{n}{2}$ and $\rho>-\frac{n}{2}+2$, while, if $n\geq 5$, we get $\delta\geq -2$ and $\rho\geq 0$. One reason for these seemingly odd behaviour is that we need, as an underlying hypotheses, both $\mathrm{Ric}_g$ and $B$ in $L^2$. In dimensions $3$ and $4$ this behaviour of the Ricci tensor is guaranteed by the condition $\delta>-\frac{n}{2}$. In contrast, for $n\geq 5$, this is not the case, and the condition $\delta\geq -2$ guarantees the desired asymptotic behaviour. The weight $\rho$ is subtlety tied to $\delta$ via the solutions of a PDE (see the proof of Theorem \ref{B-AS}).

From this general result, we will extract as corollaries an almost-Schur inequality for the scalar curvature, analogous to the one proved by De Lellis-Topping in \cite{Topping}; for $Q$-curvature, analogous to \cite{chinos}, and for the mean curvature of AE-hypersurfaces in Ricci-flat spaces, analogous to \cite{Cheng1}. Some of these results, in the compact case, where summarized and put together in \cite{Cheng2}.

\bigskip
Finally, we will analyse static potentials on AE manifolds. These structures have recieved plenty of interest and some important results, together with applications, can be consulted in J. Corvino's influential paper \cite{Corvino}. Strong rigidity of these structures has been observed as a consequence of integral geometric identities, for instance, in \cite{allan} by Barbosa \textit{et al.} and in \cite{Miao} by Miao-Tam for compact and AE manifolds respectively. In this context, we will show the validity of the following result.

\begin{thm*}
Suppose that $(M^n,g)$ is a $H_{s+3,\delta}$-AE manifold with $n\geq 3$, $s>\frac{n}{2}$ and $\delta>-1$, which admits a non-negative static potential $f$. Furthermore, suppose that $\mathrm{Ric}_g\in H_{s+1,\rho}$, for some $\rho>\frac{n}{2}-1$, and that $\partial M=f^{-1}(0)$ consists of $N$-closed connected components, labelled by $\{\Sigma_i \}_{i=1}^{N}$. Then, it follows that 
\begin{align}
\int_Mf|\mathrm{Ric}_g|^2_g\mu_g = \frac{1}{2}\sum_{i=i}^Nc_i\int_{\Sigma_i}R_{h_i}\mu_{\partial M},
\end{align}
where the constants $c_i=|\nabla f|_{\Sigma_i}$. In particular, if $\partial M=\emptyset$, then $(M^n,g)$ is isometric to $(\mathbb{R}^n,e)$, where $e$ is the euclidean metric.
\end{thm*}

This theorem, in dimension three, was proved in \cite{Miao}. We will show how their result has a natural (although non-trivial) extension to general dimensions. Furthermore, we will show how this result can be used to give very simple characterizations for the admissible topologies for the event horizon of static black holes. The characterization we will provide is not new, although it is derived using much simpler techniques than the ones used to prove more general results, such as the ones appearing in \cite{EGP} and \cite{Hawking}.  

\bigskip
With the above in mind, the organization of the paper will be as follows. We will first review some of the analytic tools needed for the core of the paper. Then, under sufficient conditions, we will establish the validity of the Pohozaev-Schoen identity on asymptotically euclidean manifolds. Finally, we will go into the applications, where we will show how some rigidity results of generalized solitons can be derived as a straightforward consequence of this identity. Furthermore, a generalized almost-Schur-type inequality, in this non-compact setting, can also be obtained using these tools. As our last application, we will show how some identities which are related with rigidity of static potentials on asymptotically euclidean manifolds, also follow directly from these principles.

\section{AE Manifolds}

In this section, the idea will be to present the main definitions and properties of AE manifolds that we will use in the subsequent sections. The analysis of these structures has been developed along the years, including the seminal works \cite{Bartnik}, \cite{Cantor}, \cite{CB1}, \cite{Maxwell-Dilts}, \cite{maxwell}, \cite{maxwell1}. We are following the classical notations and results established in \cite{CB1}, where the detailed proofs can be consulted.

\begin{defn}
A $n$-dimensional smooth Riemannian manifold $(M,e)$ is called euclidean at infinity if there exits a compact set $K$ such that $M\backslash K$ is the disjoint union of a finite number of open sets $U_i$, such that each $(U_i,e)$ is isometric to the exterior of an open ball in the Euclidean space.   
\end{defn}

On manifolds euclidean at infinity, we define $d=d(x,p)$ the distance in the Riemannian metric $e$ of an arbitrary point $x$ to a fixed point $p$. We will typically omit the dependence on $(x,p)$. 

\begin{defn}
A weighted Sobolev space $H_{s,\delta}$, with $s$ a nonnegative integer and $\delta\in\mathbb{R}$, is a space of tensor fields $u$ of some given type on the manifold $(M,e)$ Euclidean at infinity with generalized derivatives of order up to $s$ in the metric $e$ such that $D^mu(1 + d^2)^{\frac{1}{2}(m+\delta)}\in L^2$, for $0\leq m\leq s$. It is a Banach space with norm
\begin{align}\label{norm}
||u||^{2}_{H_{s,\delta}}\doteq \sum_{0\leq m\leq s}\int_M|D^mu|_e^2(1 + d^2)^{(m+\delta)}d\mu_{e}
\end{align}
where $D$ represents the $e$-covariant derivative and $\mu_e$ the Riemannian volume form associated with $e$.
\end{defn}

\begin{remark}
It is important to note that the space $C^{\infty}_0(M,E)$ of compactly supported tensor fields of some given type is dense in $H_{s,\delta}$ for all $s\in\mathbb{N}$ and $\delta\in\mathbb{R}$. 
\end{remark}

\begin{defn}\label{AE-def}
We will say that a Riemannian metric $g$ on $M$ is asymptotically euclidean if $g-e\in H_{s,\delta}$, $s>\frac{n}{2}$ and $\delta>-\frac{n}{2}$.
\end{defn}

These weighted spaces share several properties of the usual Sobolev spaces. For instance, we have the continuous embedding:

\begin{lemma}\label{sobolevembedding}
Let $(M, e)$ be a manifold euclidean at infinity. The following inclusion holds and is continuous:
\begin{align*}
H_{s,\delta}\subset C^{s'}_{\delta'}
\end{align*}
if $s' <s -n/2$ and $\delta' <\delta +n/2$.
\end{lemma}

In the above lemma, the weighted-$C^k$ spaces are equipped with the following norm:
\begin{align*}
||u||_{C^k_{\beta}}\doteq \sup_{x\in M}\sum_{0\leq l\leq k}|D^lu|_e(1+d^2_e)^{\frac{1}{2}(\beta+l)}.
\end{align*}

Also, the following multiplication property is very useful to analyse the range of differential operators acting between these weighted spaces. 

\begin{lemma}\label{sobolevmultiplication}
If $(M,e)$ is an euclidean at infinity, then the following continuous multiplication property holds
\begin{align*}
H_{s_1,\delta_1}\times H_{s_2,\delta_2}&\mapsto H_{s,\delta},\\
(f_1,f_2)&\mapsto f_1\otimes f_2,
\end{align*}
if $s_1,s_2\geq s$, $s<s_1+s_2-\frac{n}{2}$ and $\delta<\delta_1+\delta_2+\frac{n}{2}$.
\end{lemma}

\subsection*{AE manifolds with boundary}

The idea in this subsection is to extend to above ideas to AE manifolds with boundary. We will consider $(M,e)$ to be a manifold euclidean at infinity, according to the definition presented above, but allowing that within the compact region $K$ we have a boundary. Then, we have a finite number of end charts, say $\{U_i,\varphi_i \}_{i=1}^{k}$, with $\varphi(U_i)\simeq \mathbb{R}^n\backslash B_1(0)$, and a finite number of coordinate charts covering the compact region $K$, say $\{U_i,\varphi_i\}_{i=k+1}^N$. We can consider a partition of unity $\{\eta_i\}_{i=1}^N$ subordinate to the coordinate  cover $\{U_i,\varphi_i\}_{i=1}^N$, and let $V_i$ be equal to either $\mathbb{R}^n$ or $\mathbb{R}^n_{+}$, depending on whether $U_i$, $i\geq k+1$, is an interior of boundary chart respectively. Then, given a vector bundle $E\xrightarrow{\pi} M$, we can define $H_{s,\delta}(M,E)$ to be the subset of $H_{s,loc}(M,E)$ such that 
\begin{align}\label{locsob}
||u||_{H_{s,\delta}}&=\sum_{i=1}^k||{\varphi^{-1}_i}^{*}(\eta_i u)||_{H_{s,\delta}(\mathbb{R}^n)} + \sum_{i=k+1}^N||{\varphi^{-1}_i}^{*}(\eta_i u)||_{H_{s}(V_i)}<\infty. 
\end{align}

It should be noted that Lemmas \ref{sobolevembedding}-\ref{sobolevmultiplication} remain true, since they hold for both compact manifolds with boundary and for the asymptotic region.

We should now comment on the properties of the \textit{traces} of fields in $H_{s,\delta}(M)$. From the theory of Sobolev spaces on compact manifolds with boundary, we know that if $\Omega$ is a compact smooth manifold with smooth boundary $\partial\Omega$, then, for any function $u\in H_{s}(M)$, with $s>\frac{1}{2}$, we have a continuous trace map\footnote{Here we are appealing to Sobolev spaces defined for non-integer $s\in\mathbb{R}$, which we regard as a standard tool. See, for instance, \cite{Taylor} for the details.}
\begin{align*}
\tau: H_{s}(\Omega) \mapsto H_{s-\frac{1}{2}}(\partial\Omega).
\end{align*}

In the case that $u$ is actually a section of a vector bundle $E\xrightarrow{\pi} \Omega$, we get an analogous result. To see this, notice that the use of boundary charts trivializing $E$ naturally gives us an induced vector bundle $\tilde{E}$ over $\partial\Omega$ with the same typical fibre as $E$, whose vector bundle structure comes from the boundary coordinate charts, which come from the coordinate systems of $\Omega$ adapted  to the boundary that trivialize $E$. Clearly, we can both extend the sections of $\tilde{E}$ to section of $E$, and, vice versa, restrict sections of $E$ to sections of $\tilde{E}$. Hence, we will drop the tilde from $\tilde{E}$, and take this setting as implicit. Now, notice that the restriction of a section of $E$ to $\partial\Omega$ will produce a map $\tau u: \partial\Omega\mapsto E$. Furthermore, by localizing $u$ in a fashion analogous to what was done prior to introducing the norm (\ref{locsob}), we get that $u\in H_{s}(\Omega,E)$ iff given a partition of unity $\{\alpha_i \}$ subordinate to a coordinate cover $\{U_i,\varphi_i \}$, the localized fields $\alpha_i u \in H_s(U_i,E)$. We can then use boundary charts to trace these localized fields, and get, $\forall$ $s>\frac{1}{2}$, a continuous trace map \footnote{The interested reader can find the details on \cite{Palais}}
\begin{align*}
\tau:H_{s}(\Omega,E) \mapsto H_{s-\frac{1}{2}}(\partial\Omega,E).
\end{align*}

In the case of $AE$-manifolds, since the trace map only involves the compact core $K\subset\subset M$, from the above and a partition of unity argument, we get that, for $s>\frac{1}{2}$, we have continuous trace map
a continuous trace map
\begin{align}
\tau:H_{s,\delta}(M,E) \mapsto H_{s-\frac{1}{2}}(\partial M,E).
\end{align}

\begin{remark}
In the subsequent sections, in the notation $H_{s,\delta}(M,E)$, whenever the vector bundle is implicit from the context, we will simply write $H_{s,\delta}$. The same will be done with other functional spaces.
\end{remark}

\begin{remark}
We would like to point out that the definition of AE Riemannian manifold provided in Definition \ref{AE-def} follows the same lines as the definition given in \cite{CB1}-\cite{CB2} and \cite{Bartnik}, among others. In particular, the definition adopted by Bartnik in \cite{Bartnik} is related to ours by a simple shift in the weight parameter. Explicitly, $\delta_{b}=-(\delta + \frac{n}{2})$, where $\delta_b$ is the weighting parameter used in the definition of the functional spaces in \cite{Bartnik}. Nevertheless, analogous definitions can be given without invoking weighted Sobolev spaces. For instance, it is quite usual to define an AE metric on a manifold euclidean at infinity by its behaviour in each end in coordinate systems. That is, demanding $g$, near infinity, to satisfy a condition of the form
\begin{align}\label{AE-def2}
g_{ij}=\delta_{ij} + o_{k}(|x|^{-\tau}),
\end{align}
where $o_{k}(|x|^{-\tau})$ implies that $\partial^{l}g_{ij}=o(|x|^{-\tau-l})$ for all $1\leq l \leq k$ and $\tau>0$. It should be completely clear that the condition expressed in Definition \ref{AE-def} controls the behaviour of $g$ and its derivatives at infinity and implies that the metric satisfies a conditions of the form of (\ref{AE-def2}). Also, similar definitions could be given by means of weighted $C^k$-spaces and H\"older spaces. All of these definitions are very easily related by means of embedding theorems and analysis of the behaviour of the field at infinity, and the precise choice relies on the specific problems to be treated and which of them seems to be more transparently tailored to deal with such problems.
\end{remark}

\section{Pohozaev-Schoen identity on AE manifolds}

The idea now will be to try to deduce the integral identity in the context of AE manifolds. We will consider that such manifolds $M$ may have compact boundary $\partial M$ as described in the previous section.


\begin{thm}\label{thm1}
Let $(M,g)$ be a $H_{s,\delta}$-asymptotically euclidean manifold, with $s>\frac{n}{2}+1$ and $\delta>-\frac{n}{2}$. Also, let $B$ be a symmetric $(0,2)$-tensor field, satisfying $\mathrm{div}_gB=0$, and let $X$ be a vector field on $M$. Suppose $X\in C^1$ is bounded with $\nabla X\in L^2$. Furthermore, suppose that $X\in L^2_{-(\rho+1)}$ and $B\in H_{s+1,\rho}$, with $\rho\geq 0$. Then, the following equality holds:
\begin{align}\label{PS}
\int_M\langle \overset{\circ}{B},\pounds_X g \rangle\mu_g   = \frac{2}{n} \int_M X(\mathrm{tr}_gB)\mu_g + 2 \int_{\partial M}\overset{\circ}{B}(X,\nu)\mu_{\partial M}
\end{align}
\end{thm}

\begin{proof}

With a slight abuse of notation, making contact with standard notation for differential forms, we will denote the contraction of $B$ and $X$ by $i_XB\doteq B(X,\cdot)$.  Now, if we consider $B\in C^{\infty}_0$ and $X$ satisfying the hypotheses described above, then the following is standard:
\begin{align}\label{PS1}
\frac{1}{2}\int_M\langle B,\pounds_{X}g\rangle\mu_g=\int_M\mathrm{div}_g(i_XB)\mu_g=\int_{\partial M}B(X,\nu)\mu_{\partial M},
\end{align}
where $\nu$ denotes the unit normal of $\partial M$. Now, consider $\{B_n\}_{n=1}^{\infty}\subset C^{\infty}_0(M)$, such that 
\begin{align*}
B_n&\xrightarrow{H_{s+1,\rho}} B.
\end{align*}
Then, we have the following:
\begin{align*}
\Big\vert \int_M(\langle B,\pounds_{X}g\rangle-\langle B_n,\pounds_{X}g\rangle)\mu_g\Big\vert&\leq \int_M|\langle B-B_n,\pounds_{X}g\rangle|\mu_g,\\
&\leq \int_M|B-B_n|_g|\pounds_{X}g|_g\mu_g,\\
&\leq ||B-B_n||_{L^2}||\pounds_{X}g||_{L^2}\xrightarrow[n\rightarrow\infty]{} 0,
\end{align*}
where we have made use of the fact that, since $\nabla X\in L^2$, so does $\pounds_{X}g$. All this gives us 
\begin{align}\label{PS2}
\frac{1}{2}\int_M\langle B,\pounds_{X}g\rangle\mu_g=\lim_{n\rightarrow\infty}\int_{\partial M}B_n(X,\nu)\mu_{\partial M}.
\end{align}
Now, a similar argument gives us
\begin{align*}
\Big\vert\int_{\partial M}(B-B_n)(X,\nu)\mu_{\partial M}\Big\vert\lesssim \int_{\partial M} |B-B_n|_g|X|_g\mu_{\partial M}
\end{align*}
where  we have used that, given two tensor fields $A,B$ on $M$ then the following pointwise estimate holds $|C(A\otimes B)|_g\leq K(n,d_1,d_2)|A|_g|B|_g$, where $C$ denotes an arbitrary contraction and $K$ is a constant depending only on the dimension $n$ of $M$ and the dimensions $d_1$ and $d_2$ of the fibres. Notice that, since we have a continuous trace $\tau: H_{s+1,\rho}(M)\mapsto H_{s}(\partial M)$, then $\tau(B-B_n)\xrightarrow[n\rightarrow \infty]{} 0$ in $L^2(\partial M)$. Hence, from the above, we get
\begin{align*}
\Big\vert\int_{\partial M}(B-B_n)(X,\nu)\mu_{\partial M}\Big\vert\lesssim ||B-B_n||_{L^2(\partial M)} ||X||_{L^2(\partial M)}\xrightarrow[n\rightarrow \infty]{} 0
\end{align*}
Putting this together with (\ref{PS2}), we get
\begin{align}\label{PS3}
\frac{1}{2}\int_M\langle B,\pounds_{X}g\rangle\mu_g=\int_{\partial M}B(X,\nu)\mu_{\partial M}.
\end{align}

Now, if we write $B=\overset{\circ}{B}+\frac{1}{n}\mathrm{tr}_gB \: g$, where $\overset{\circ}{B}$ is the traceless part of $B$, we get the following
\begin{align}\label{PS4}
\langle B, \pounds_Xg \rangle &= \langle \overset{\circ}{B},\pounds_X g \rangle + \frac{2}{n}\mathrm{tr}_gB\mathrm{div}_gX.
\end{align}
Notice that $\mathrm{tr}_gB\mathrm{div}_gX=\mathrm{div}_g(\mathrm{tr}_gB X) - X(\mathrm{tr}_gB)$, which gives us that 
\begin{align*}
\int_M\mathrm{tr}_gB\mathrm{div}_gX\mu_g=\int_{\partial M}\mathrm{tr}_gB\langle X,\nu\rangle \mu_{\partial M} -\int_M X(\mathrm{tr}_gB)\mu_g \;\; \forall \;\; B\in C^{\infty}_{0}.
\end{align*}
Doing as above and approximating $B$ by compactly supported fields converging to $B$ in $H_{s+1,\rho}$, we get
\begin{align*}
\Big\vert\int_M(\mathrm{tr}_gB\mathrm{div}_gX-\mathrm{tr}_gB_n\mathrm{div}_gX)\mu_g\Big\vert&\leq \int_M|\mathrm{tr}_g(B-B_n)||\mathrm{div}_gX| \mu_g,\\
&\lesssim \int_M|B-B_n|_g|\nabla X|_g\mu_g,\\
&\leq ||B-B_n||_{L^2}||\nabla X||_{L^2}\xrightarrow[n\rightarrow\infty]{} 0,
\end{align*}
since, by hypothesis, $\nabla X\in L^2(M)$. The above shows that 
\begin{align}
\int_M\mathrm{tr}_gB\mathrm{div}_gX\mu_g=\lim_{n\rightarrow\infty}\Big\{ \int_{\partial M}\mathrm{tr}_gB_n\langle X,\nu\rangle \mu_{\partial M} -\int_M X(\mathrm{tr}_gB_n)\mu_g \Big\}.
\end{align}
Furthermore, notice that 
\begin{align*}
\Big\vert\int_M X(\mathrm{tr}_g(B-B_n))\mu_g\Big\vert&\leq \int_M|\langle \nabla\mathrm{tr}_g(B-B_n),X\rangle|\mu_g,\\
&\lesssim \int_M|\nabla(B-B_n)|_g|X|_g\mu_g,\\
&=\int_M|\nabla(B-B_n)|_g(1+d^2_e)^{\frac{1}{2}(\rho+1)}(1+d^2_e)^{-\frac{1}{2}(\rho+1)}|X|_g\mu_g,,\\
&\leq ||X||_{L^2_{-(\rho+1)}}||\nabla(B-B_n)||_{L^2_{\rho+1}}\xrightarrow[n\rightarrow\infty]{} 0
\end{align*}

Similarly, we also get that
\begin{align*}
\Big\vert\int_{\partial M}\mathrm{tr}_g(B - B_n)\langle X,\nu\rangle \mu_{\partial M}\Big\vert&\lesssim \int_{\partial M}|B - B_n|_g |X|_g,\\
&\leq ||B-B_n||_{L^2(\partial M)} ||X||_{{L^2(\partial M)}} \xrightarrow[n\rightarrow\infty]{} 0,
\end{align*}
where we have again used the continuity of the trace map. All of the above, gives us that
\begin{align}
\int_M\mathrm{tr}_gB\mathrm{div}_gX\mu_g=\int_{\partial M}\mathrm{tr}_gB\langle X,\nu\rangle \mu_{\partial M} -\int_M X(\mathrm{tr}_gB)\mu_g.
\end{align}
Thus, integrating (\ref{PS4}) and applying the above identity, we see that the following holds
\begin{align*}
\int_M\langle B, \pounds_Xg \rangle \mu_g = \int_M\langle \overset{\circ}{B},\pounds_X g \rangle\mu_g  +  \frac{2}{n}\int_{\partial M}\mathrm{tr}_gB\langle X,\nu\rangle \mu_{\partial M} - \frac{2}{n} \int_M X(\mathrm{tr}_gB)\mu_g.
\end{align*}
Putting this together with (\ref{PS3}), we get that the Pohozaev-Schone equality is valid in the present context:
\begin{align}
\int_M\langle \overset{\circ}{B},\pounds_X g \rangle\mu_g   = \frac{2}{n} \int_M X(\mathrm{tr}_gB)\mu_g + 2 \int_{\partial M}\overset{\circ}{B}(X,\nu)\mu_{\partial M}
\end{align}

\end{proof}

\section{Applications} \label{applications}

Before going into the applications of the above identity, we would like to state and give an independent proof of the following rigidity statement linked with AE-manifolds. It should be stressed that similar statements, with slightly different hypotheses, can be found in several papers, such as \cite{Lee-Parker},\cite{Bartnik},\cite{volcomp}.

\begin{thm}\label{Ricciflatrigidity}
Let $(M,e)$ be a, $n$-dimensional manifold euclidean at infinity, with $n\geq 3$, and consider a $H_{s,\delta}$-AE metric $g$, with $s> \frac{n}{2}+2$ and $\delta>-\frac{n}{2}$. Then, if $g$ Ricci-flat, it holds that $(M,g)$ is isometric to $\mathbb{R}^n$ with its standard flat metric.  
\end{thm} 

\begin{proof}
This proof will be carried out via a number of well-known results. First, notice that under our hypotheses on $g$, its Ricci tensor can be written, on any particular end, in harmonic coordinates giving us an expression of the form
\begin{align}
{\mathrm{Ric}_g}_{ij}=-g^{ab}\partial_a\partial_bg_{ij} + Q_{ij}(g,\partial g),
\end{align}
where $Q(g,\partial g)$ is a quadratic polynomial in $\partial g$. Thus, in these coordinates, the operator appearing in the right hand side of the above identity is an elliptic second order operator acting on the metric component functions written in such coordinates. Nevertheless, since the left-hand side is a tensor, we get that if $\mathrm{Ric}_g\in L^2_{\tau}$, so is the left hand side on any given coordinate system. Thus, via elliptic theory we can improve the decay of the metric. This can be achieved by appealing to Theorem 5.1 in \cite{CB1} plus an inductive argument of the type of Lemma 3.6 in Appendix II of \cite{CB2}. In any case, the final result is that we can guarantee that $g-e\in H_{2,\tau-2}$, as long as $\tau-2<\frac{n}{2}-2$. In particular, if $g$ is Ricci-flat, then $g-e\in H_{2,\rho}\cap C^2$ for any $\rho<\frac{n}{2}-2$. This implies that, near infinity, $g_{ij}=\delta_{ij}+o_2(|x|^{-(\rho+\frac{n}{2})})$. Thus, we get that
\begin{align*}
g_{ij}=\delta_{ij}+o_2(|x|^{-\rho'}), \;\; \forall \;\; \rho'<n-2.
\end{align*}
The above, plus Ricci-flatness, imply that the ADM mass of $g$ is actually well-defined, and, in particular, we are under the decay assumptions of Theorem 2.1 of \cite{Miao-ADM}. Hence, we can rewrite the ADM mass of $g$, say $m(g)$, in terms of its Einstein tensor \textit{at infinity}. Thus, Ricci-flatness implies that $m(g)=0$. From this condition we get that $(M,g)$ is isometric to $(\mathbb{R}^n,\cdot)$ via the well-known rigidity of the positive mass theorem, which can be obtained in this setting via Proposition 10.2 in \cite{Lee-Parker}.
\end{proof}

\subsection*{Generalized Solitons}

In this section, we will consider $B$-\textit{generalzed solitons} and use the Pohozaev-Schoen identity (\ref{PS2}) to obtain some rigidity results. To begin with, we will consider an asymptotically euclidean manifold $(M,g)$, a symmetric tensor field $B$ and a vector field $X$. Then, we say that $X$ defines a $B$-generalized soliton structure on the AE-manifold $M$ if the following equation is satisfied
\begin{align}\label{Bgensol}
B+\frac{1}{2}\pounds_Xg=\lambda g,
\end{align}
for some function $\lambda$. Then, we have the following immediate lemma.
\begin{thm}\label{Bgensol}
Let $(M^n,g)$ be an  asymptotically euclidean manifold without boundary, $n\geq 3$, where $g$ satisfies the hypotheses of Theorem \ref{thm1}, and suppose that the tensor field $B\in H_{s+1,0}$, $s>\frac{n}{2}$, is $g$-divergence-free. Furthermore, let $X\in H_{s+2,-1}$ and $\mathrm{tr}_gB=0$. Then, under these assumptions, any $B$-generalized soliton is $B$-flat.
\end{thm}
\begin{proof}
Under our working hypotheses, since $H_{s+2,-1}\hookrightarrow C^1_{-1}$, we have both that $X\in H_{1,-1}\cap C^1$. Thus, we know that the identity (\ref{PS2}) holds. Hence, if $\mathrm{tr}_gB=0$, we have that
\begin{align*}
\int_M\langle \overset{\circ}{B},\pounds_X g \rangle\mu_g=0.
\end{align*}
Furthermore, since $X$ satisfies (\ref{Bgensol}), the following identities holds
\begin{align*}
\langle \overset{\circ}{B},\pounds_X g \rangle &= \langle B,\pounds_{X,conf} g \rangle,\\
&=- \frac{1}{2}\langle \pounds_Xg,\pounds_{X,conf} g \rangle,\\
&=- \frac{1}{2}\langle \pounds_{X,conf}g,\pounds_{X,conf} g \rangle,
\end{align*}
where in the first line we have introduced the conformal Killing Laplacian, defined by $\pounds_{X,\text{conf}}g\doteq \pounds_Xg-\frac{2}{n}\mathrm{div}_gX g$, which defines a traceless tensor; in the second line we used the soliton equation (\ref{Bgensol}); and in the last line we used this traceless property of the conformal Killing Laplacian. With the above, we see that (\ref{Bgensol}) implies that 
\begin{align*}
\int_M|\pounds_{X,conf}g|_g^2\mu_g=0,
\end{align*}
which is equivalent to $X\in H_{s+1,-1}$ being a conformal Killing field of $g$. But it is known that there are no non-trivial Killing fields with this asymptotic behaviour \cite{Chris}-\cite{maxwell}, thus $X=0$, and thus $B$-flat.
\end{proof}

From this theorem, we get the following corollaries.

\begin{coro}
Consider an $n$-dimensional AE-Riemannian manifold $(M,g)$ without boundary, with $g-e\in H_{s+3,\delta}$, with $s>\frac{n}{2}$ and $\delta>-\min\{\frac{n}{2},2\}$ and $n\geq 3$. Then, if $R_g=0$, any Ricci almost soliton $X\in H_{s+2,\delta+1}$ on $(M,g)$ is isometric to $(\mathbb{R}^n,e)$, with $e$ the euclidean metric.
\end{coro}
\begin{proof}
First, recall a Ricci almost soliton structure on a Riemannian manifold $(M,g)$ is given by a vector field $X$ satisfying the equation
\begin{align}\label{Riccisol1}
Ric_g+\frac{1}{2}\pounds_Xg=\lambda g,
\end{align}
for some function $\lambda$. Notice that the above equation is equivalent to the $G$-generalized soliton equation given by
\begin{align}\label{Riccisol2}
G_g +\frac{1}{2}\pounds_Xg=(\lambda - \frac{1}{2}R_g)g.
\end{align}
where $G_g\doteq \mathrm{Ric}_g-\frac{1}{2}R_gg$ stands for the Einstein tensor associated to the metric $g$. Notice that under our hypotheses for $g$, we have that $G_g\in H_{s+1,0}$. Thus, from this fact and the hypotheses on $X$, we are under the hypotheses of Theorem \ref{Bgensol}. Hence, if $\mathrm{tr}_gG_g=\frac{2-n}{2}R_g=0$, we have that $Ric_g=0$ as a consequence of the Theorem. Thus, from Theorem \ref{Ricciflatrigidity}, we get that $(M,g)$ is isometric to $(\mathbb{R}^n,e)$.

\end{proof}

Similarly, we can consider \textit{Codazzi-generalized solitons} for some AE-hypersurface $(M^n,g)$ isometrically immersed in a Ricci-flat space $(\tilde{M}^{n+1},\tilde{g})$. In such a case, from Codazzi's equation, we know that the tensor $\pi\doteq K - \mathrm{tr}_gK g$ satisfies the conservation equation
\begin{align}\label{piconservation}
\mathrm{div}_g\pi=0,
\end{align}
where $K$ stands for the \textit{extrinsic curvature} of the hypersurface $M\hookrightarrow \tilde{M}$, which is defined as $K\doteq \langle \RN{2},n \rangle$, where $\RN{2}$ stands for the second fundamental form associated to the immersion, and $n$ for the unit normal field. Now, we define a Codazzi-generalized soliton structure on such an immersion as the choice of a vector field $X$ on $M$ satisfying the following soliton-type equation
\begin{align}\label{codsol}
K + \frac{1}{2}\pounds_Xg = \lambda g,
\end{align} 
for some function $\lambda$ on $M$. 

\begin{coro}
Consider the setting described above, suppose that $M$ does not have a boundary, and suppose that $(M,g,X)$ provides a Codazzi-generalized soliton structure in a Ricci-flat space. Furthermore, suppose that $K\in H_{s+1,0}$ and $X\in H_{s+2,-1}$. Then, if $M$ is a minimal or maximal hypersurface\footnote{We are enabling $\tilde{g}$ to be either Riemannian or Lorentzian.}, that is, $\mathrm{tr}_gK=0$, then $(M,g)$ is a totally geodesic scalar-flat hypersurface. In particular, if the ambient space is actually flat, then $(M,g)$ is isometric to $(\mathbb{R}^n,e)$, with $e$ the euclidean metric.
\end{coro}
\begin{proof}
Since our ambient manifold is Ricci-flat, we know that the Codazzi equation for hypersurfaces implies that (\ref{piconservation}) holds. Furthermore, the hypothesis of $\mathrm{tr}_gK=0$ implies that $\pi=K$, and thus $K$ is a \textit{conserved} symmetric tensor, which satisfies all the hypotheses demanded for the tensor field $B$ in Theorem \ref{thm1}. Furthermore, following the same arguments as in the previous corollary, we have that $X$ also satisfies all the hypotheses imposed in Theorem \ref{thm1}. Thus, using Theorem \ref{Bgensol}, we immediately get that $M\hookrightarrow\tilde{M}$ is totally geodesic. Now, notice that the Gauss equation can be written en the following way. Given $X,Y,Z,W\in\Gamma(TM)$, we have that
\begin{align}\label{Gauss}
\tilde{g}(\tilde{R}(X,Y)Z,W)=g(R(X,Y)Z,W) + \epsilon\big(  K(X,Z)K(Y,W) - K(Y,Z)K(X,W) \big),
\end{align}
where $\epsilon=\tilde{g}(n,n)=\pm 1$, depending on whether the normal direction is \textit{space-like} or \textit{time-like}. Now, we have just shown that $M$ is totally geodesic, thus $K=0$. Thus, consider an orthonormal frame $\{E_{\alpha}\}_{\alpha=0}^n$, where $E_{0}=n$ and thus $E_i$ is tangent to $M$ $\forall$ $i=1,\cdots,n$, so that we can write
\begin{align*}
\mathrm{Ric}_g(X,Y)&=\sum_{i=1}^n\tilde{g}(\tilde{R}(E_i,Y)Z,E_i),\\
&=\mathrm{Ric}_{\tilde{g}}(X,Y)-\epsilon\tilde{g}(\tilde{R}(n,Y)Z,n),\\
&=-\epsilon\tilde{g}(\tilde{R}(n,Y)Z,n),
\end{align*} 
where in the third line we have used the hypotheses that $\tilde{g}$ is Ricci-flat. From this, we get
\begin{align*}
R_g=-\epsilon\sum_{j=1}^n\tilde{g}(\tilde{R}(n,E_j)E_j,n)=-\epsilon\sum_{j=1}^n\tilde{g}(\tilde{R}(E_j,n)n,E_j)=-\epsilon\mathrm{Ric}_{\tilde{g}}(n,n)=0,
\end{align*}
which proves that $(M,g)$ is scalar-flat. Finally, since $K=0$, it is a trivial consequence of (\ref{Gauss}) that if $\mathrm{Riem}_{\tilde{g}}=0$, then $\mathrm{Riem}_{g}=0$ and thus $(M^n,g)$ is isometric to $(\mathbb{R}^n,e)$.

\end{proof}
\begin{remark}
In order to gain some insight about the kind of geometric picture that a Codazzi-soliton represents, it is an interesting observation to note that, following closely Al\'ias-Lira-Rigoli \cite{ALR}, we can understand (\ref{codsol}) following the kind of definition provided in \cite{ALR} for mean curvature flow solitons. That is, if, given an immersion $\psi:M^n\mapsto \tilde{M}^{n+1}$, we say it defines a Codazzi-soliton with respect to a conformal Killing field $X\in\Gamma(T\tilde{M})$ if it satisfies
\begin{align*}
X^{\perp}=-n,
\end{align*}
along $\psi$, where $X^{\perp}$ denotes the orthogonal component of $X$ to $\psi(M)$ and $n$ denotes the unit normal of $\psi(M)$. Then, following the same line of argument presented in \cite{ALR}, equation (\ref{codsol}) is an equation that any Codazzi-soliton must satisfy. Thus, in this sense, Codazzi solitons stand in perfect analogy with interesting soliton-like equations related to interesting geometric flows.  

\end{remark}

\begin{remark}
Taking into account the above remark and corollary, it is also interesting to note that, in the case of manifolds euclidean at infinity, admitting a scalar-flat metric is not a trivial statement. This can be viewed as a consequence of some important results, concerning for instance the Yamabe classification for AE manifolds \cite{Maxwell-Dilts} and topological obstructions related with the positive mass theorem \cite{Schoen-Yau}. More explicitly, using the Yamabe classification provided in \cite{Maxwell-Dilts}, we know that AE manifolds produced by removing points from closed manifolds which do not admit Yamabe positive metric do not carry any metric of zero scalar curvature. Using the results established in \cite{Schoen-Yau}, it is possible to produce a wide variety of closed manifolds which do not admit a Yamabe positive metric. For instance, their main claim is that manifolds of the form $M^n\# T^n$, with $M^n$ closed, lie in this category. Now, all this implies that AE manifolds obtained by removal of a finite number of points from such closed manifolds do not admit any Codazzi-soliton structure in any Ricci-flat space. 
\end{remark}

Similarly to what we have done above, we could easily define new soliton-type equation related with other conserved second rank tensor fields, which we could relate to other generalized geometric flows. 

\subsection*{Generalized Almost-Schur type inequality}

The classical Schur's lemma says that every connected Einstein manifold of dimension $n\geq 3$ has constant scalar curvature. In \cite{Topping}, De Lellis and Topping studied a quantitative version of the classical Schur Lemma, that allows to infer that, if a closed Riemannian manifold is close to being Einstein in the $L^{2}$ sense, then its scalar curvature is close to a constant in $L^{2}$. In this scope, they proved the so called ``Almost Schur Lemma" that claims if $(M^{n},g)$ is a closed Riemannian manifold of dimension $n\geq 3$, with nonnegative Ricci curvature, then

$$\displaystyle\int_{M}(R-\bar{R})^{2}\mu_{g}\leq \frac{4n(n-1)}{(n-2)^{2}}\displaystyle\int_{M}\left|Ric-\frac{R}{n}g\right|^{2}\mu_{g},$$
where we denote $\bar{R}=\frac{1}{Vol M}\int_{M}R\mu_{g}$. Furthermore, the equality occurs if, and only if, $M$ is an Einstein manifold.

We observe that a central step in the proof is the following integral identity (see \cite{Topping})

$$-\displaystyle\int_{M}\langle dR, df\rangle \mu_{g}=\frac{2n}{n-2}\displaystyle\int_{M}\langle{\stackrel{\circ}{Ric}},D^{2}f\rangle,$$
where $f$ is a solution of a PDE. In fact, this integral formula is the Pohoz\v{a}ev-Sch\"oen identity, if $X=\nabla f$ and $\partial M=\emptyset$. In fact, this idea can be used to obtain these type of inequalities for more general tensors and under others suitable conditions on the Ricci curvature (\cite{Barbosa},\cite{Cheng1},\cite{Cheng2}). The idea of this section is to generalize these type of inequalities obtained in \cite{Topping} to non-compact manifolds in our case of interest. Thus, suppose that $(M^n,g)$ is an $H_{s+3,\delta}$-asymptotically euclidean manifold, with $n\geq 3$, $\delta>-\frac{n}{2}$, $\delta\geq -2$ and $s>\frac{n}{2}$. Recall from \cite{CB1} that $\Delta_g:H_{s+2,\delta_0}\mapsto H_{s,\delta_0+2}$ is an isomorphism for $-\frac{n}{2}<\delta_0<\frac{n}{2}-2$. It will be an interesting observation to note that this is not true for closed manifolds, where what we have is that the operators of the form $\Delta_g-a:H_{s+2}\mapsto H_s$ with $a>0$ and $a\in H_{s}$  have this isomorphism property. This difference can be traced back to the fact that on closed manifolds constant functions are in the Kernel of $\Delta_g$, while on AE-manifolds we have a Poincar\'e inequality \cite{Bartnik}-\cite{Maxwell-Dilts}, which guarantees that $\Delta_g$ has the isomorphism property. In fact, a stronger statement that can be made. That is, given a $H_{s+1,\delta}$-AE metric, with $s>\frac{n}{2}$ and $\delta>-\frac{n}{2}$, then the Laplacian operator associated the $g$, that is $\Delta_g$, is injective on $H_{2,\rho}$ for any $\rho>-\frac{n}{2}$. This result is a consequence of Corallary 4.3 presented in Appendix 2 of \cite{CB2} and an application of the Sobolev embedding theorems. These points are key features that will allow us to produce an almost-Schur-type inequality without any restrictions on the Ricci-curvature. It is known that, for $n\geq 5$, being bounded from below is a necessary condition in the case of closed manifolds \cite{Topping}. 

\begin{thm}\label{B-AS}
Let $(M^n,g)$ is an $H_{s+3,\delta}$-asymptotically euclidean manifold without boundary, with $n\geq 3$, $s>\frac{n}{2}$, $\delta>-\frac{n}{2}$ and $\delta\geq -2$. Now, let $B$ be a symmetric $(0,2)$-tensor field satisfying the hypotheses of Theorem \ref{thm1}, that is, $\mathrm{div}_gB=0$ and $B\in H_{s+1,\rho}$, with $\rho\geq 0$, furthermore, suppose that $\rho>-\frac{n}{2} + 2$, and denote its trace by  $b\doteq \mathrm{tr}_gB$.  Then:\\
I)  there is a constant $C$ independent of the specific choice of such tensor $B$, such that the following inequality holds
\begin{align}\label{AS}
||b||_{L^2}\leq n\left( \frac{n-1}{n} + C||\mathrm{Ric}_g||_{C^0_2}\right)^{\frac{1}{2}} ||\overset{\circ}{B}||_{L^2},
\end{align}
II) If $\mathrm{Ric}_g$ is non-negative, \textit{i.e}, $\mathrm{Ric}_g(X,X)\geq 0$ $\forall$ $X\in\Gamma(TM)$, then
\begin{align}\label{AS.01}
||b||_{L^2}\leq \sqrt{n(n-1)} ||\overset{\circ}{B}||_{L^2}.
\end{align}
\end{thm}

\begin{proof}
Under the above hypotheses, since $\delta>-\frac{n}{2}$, we have that $b\in H_{s+1,\rho}$. Now, assume that $\rho<\frac{n}{2}$. If actually $\rho\geq \frac{n}{2}$, then we would have a continuous embedding $H_{s+1,\rho}\hookrightarrow H_{s+1,\sigma}$, for any $\sigma<\frac{n}{2}$. Thus, this assumption posses no actual restrictions, and from now on, so as to avoid introducing one further weighting parameter, we will assume it. Also, from the hypotheses bounding $\rho$ from below, we have that $-\frac{n}{2}+2<\rho<\frac{n}{2}$, for $n=3,4$; and $0\leq\rho<\frac{n}{2}$, for $n\geq 5$. Notice that these conditions can always be satisfied, and that, as stated before, we are not posing any further restrictions. 

Now, consider $\Delta_g:H_{s+3,\rho-2}\mapsto H_{s+1,\rho}$, and consider the equation
\begin{align}\label{AS1}
\Delta_gf=b.
\end{align}
Since $-\frac{n}{2}<\rho-2<\frac{n}{2}-2$ for any $n\geq 3$, then, from the discussion preceding the theorem, we know that there is a solution $f\in H_{s+3,\rho-2}$. Hence, we get the following:
\begin{align}\label{AS2}
\begin{split}
\int_M b^2\mu_g&=\int_Mb\Delta_gf\mu_g,\\
&=-\int_M\nabla f(b)\mu_g + \int_M\mathrm{div}_g(b\nabla f)\mu_g.
\end{split}
\end{align} 
Now, to see that the last term actually does not produce any contribution, the procedure is standard: Consider sequences $\{b_n\}\subset C^{\infty}_0$ of compactly supported functions converging to $b$ in $H_{s+1,\rho}$ and $\{\nabla f_n\}\subset C^{\infty}_0$ converging to $\nabla f$ in $H_{s+2,\rho-1}$. Then, consider the difference
\begin{align}\label{AS3}
\begin{split}
\Big\vert\int_M\{\mathrm{div}_g(b\nabla f)-\mathrm{div}_g(b_n\nabla f_n)\}\mu_g\Big\vert\leq &\int_M|\mathrm{div}_g(b(\nabla f-\nabla f_n)|\mu_g \\
&+ \int_M|\mathrm{div}_g((b-b_n)\nabla f_n)|\mu_g.
\end{split}
\end{align}
We also have that $\mathrm{div}_g(b(\nabla f-\nabla f_n))=\langle\nabla b,\nabla f-\nabla f_n\rangle + b\:\mathrm{div}_g(\nabla f-\nabla f_n)$, which gives point wise estimates of the form:
\begin{align*}
|\mathrm{div}_g(b(\nabla f-\nabla f_n))|\lesssim |\nabla f-\nabla f_n|_g|\nabla b|_g + |b||\nabla(\nabla f-\nabla f_n)|_g.
\end{align*}
Notice that, since $\rho\geq 0$, then $H_{s+3,\rho-1}\hookrightarrow L^2_{-1}$ and $H_{s,\rho+1}\hookrightarrow L^2_{1}$. Thus, after integration, we get the following estimate:
\begin{align*}
\int_M|\mathrm{div}_g(b(\nabla f-\nabla f_n)|\mu_g&\lesssim ||\nabla f-\nabla f_n||_{L^2_{-1}}||\nabla b||_{L^2_1} \\
&+ ||b||_{L^2}|||\nabla(\nabla f - \nabla f_n)||_{L^2}\xrightarrow[n\rightarrow\infty]{} 0.
\end{align*}
A similar argument can be applied to the second term in the right-hand side of (\ref{AS3}), showing that
\begin{align*}
\int_M\mathrm{div}_g(b\nabla f)\mu_g = \lim_{n\rightarrow\infty} \int_M\mathrm{div}_g(b_n\nabla f_n)\mu_g=0.
\end{align*} 
Thus, going back to (\ref{AS2}), we get the following:
\begin{align}\label{AS4}
\begin{split}
\int_M b^2\mu_g&=-\int_M\nabla f(b)\mu_g,\\
&=-\frac{n}{2}\int_M\langle \overset{\circ}{B},\pounds_{\nabla f}g \rangle\mu_g,\\
&=-n\int_M\langle \overset{\circ}{B},\nabla^2f \rangle\mu_g,\\
&=-n\int_M\langle \overset{\circ}{B},\nabla^2f - \frac{1}{n}g\Delta_gf \rangle\mu_g
\end{split}
\end{align}
where in the second line we have used (\ref{PS}). Clearly, since the left hand side is non-negative, then we have that
\begin{align}\label{AS5}
\begin{split}
0\leq -n\int_M\langle \overset{\circ}{B},\nabla^2f - \frac{1}{n}g\Delta_gf \rangle\mu_g &= n\Big\vert\int_M\langle \overset{\circ}{B},\nabla^2f - \frac{1}{n}g\Delta_gf \rangle\mu_g\Big\vert,\\
&\leq n\int_M |\overset{\circ}{B}|_g |\nabla^2f - \frac{1}{n}g\Delta_gf|_g \mu_g,\\
&\leq n||\overset{\circ}{B}||_{L^2}||\nabla^2f - \frac{1}{n}g\Delta_gf||_{L^2}.
\end{split}
\end{align}
Notice that due to the choice of our functional spaces the right-hand side is well-defined, since both $|\overset{\circ}{B}|_g,|\nabla^2f|_g\in L^2$. Also, since $H_{s+3,\rho-2}\subset C^{3}$, we get the usual Ricci identity
\begin{align*}
{\mathrm{Ric}_g}_{uj}\nabla^uf=\nabla_i\nabla_j\nabla^if - \nabla_j(\Delta_gf),
\end{align*}
which implies 
\begin{align}\label{ricciid}
{\mathrm{Ric}_g}(\nabla f,\nabla f)=\langle \nabla f,\mathrm{div}_g\left(\nabla^2 f\right)\rangle_g - \langle \nabla f,\nabla(\Delta_gf)\rangle_g.
\end{align}
Since $\nabla f\in H_{s+2,\rho-1}$ and $\mathrm{Ric}_g\in H_{s+1,\delta+2}$, with $\delta> -\frac{n}{2}$ and $\rho\geq 0$, then ${\mathrm{Ric}_g}(\nabla f,\nabla f)\in H_{s+1,\sigma}$ for any $\sigma< \delta + \rho + \frac{n}{2}$, which implies ${\mathrm{Ric}_g}(\nabla f,\nabla f)\in L^2$. Furthermore, since $\nabla^2f\in H_{s+1,\rho}$, then both $\nabla^3f$ and $\nabla(\Delta_gf)$ lie in $H_{s,\rho+1}\hookrightarrow L^{2}_1$. Thus, since $\nabla f\in L^2_{-1}$, then, both terms in the right hand side of (\ref{ricciid}) are actually in $L^1$, since 
\begin{align*}
\Big\vert \int_M\langle \nabla f,\mathrm{div}_g\nabla^2 f\rangle_g\mu_g\Big\vert&\leq \int_M |\nabla f| |\mathrm{div}_g\nabla^2 f| \mu_g \lesssim ||\nabla f||_{L^2_{-1}}||\nabla^3f||_{L^2_{1}}<\infty,\\
\Big\vert \int_M\langle \nabla f,\nabla(\Delta_gf)\rangle_g\mu_g\Big\vert&\lesssim ||\nabla f||_{L^2_{-1}}||\nabla (\Delta_gf)||_{L^2_{1}}<\infty.
\end{align*}
This implies that we can integrate (\ref{ricciid}) over $M$. Then, we can integrate by parts, so as to get
\begin{align*}
\int_M\langle \nabla f,\mathrm{div}_g\nabla^2 f\rangle_g\mu_g= - \int_M|\nabla^2f|^2_g\mu_g + \int_M\mathrm{div}_g\left(i_{\nabla f}\nabla^2f \right)\mu_g.
\end{align*}
Notice that from the above estimates, approximating $f$ in $H_{s+3,\rho-2}$ by $\{f_n \}_{n=1}^{\infty}\subset C^{\infty}_0$, we get that
\begin{align*}
\Big\vert \int_M\left(\langle \nabla f,\mathrm{div}_g\nabla^2 f \rangle_g - \langle \nabla f_n,\mathrm{div}_g\nabla^2 f_n\rangle_g\right)\mu_g\Big\vert&\lesssim ||\nabla (f-f_n)||_{L^2_{-1}}||\nabla^3f||_{L^2_{1}} \\
&+ ||\nabla f_n||_{L^2_{-1}}||\nabla^3(f-f_n)||_{L^2_{1}}\xrightarrow[n\rightarrow\infty]{}0,
\end{align*}
which implies that
\begin{align*}
\int_M\langle \nabla f,\mathrm{div}_g\nabla^2 f \rangle_g\mu_g=-\lim_{n\rightarrow\infty}\int_M|\nabla^2f_n|^2_g\mu_g=-\int_M|\nabla^2f|^2_g\mu_g,
\end{align*}
where the last equality comes from the same kind of approximation argument, plus the fact that $\nabla^2f\in L^2$. Following the lines as described above, we also get that
\begin{align*}
\int_M\langle \nabla f,\nabla(\Delta_gf)\rangle_g\mu_g=-\lim_{n\rightarrow\infty}\int_{M}|\Delta_gf_n|^2\mu_g=-\int_{M}|\Delta_gf|^2\mu_g.
\end{align*}
Thus, finally, we get that the following expression, familiar for compact manifolds, holds under our hypotheses
\begin{align}
\int_M{\mathrm{Ric}_g}(\nabla f,\nabla f)\mu_g= \int_{M}|\Delta_gf|^2\mu_g - \int_M|\nabla^2f|^2_g\mu_g.
\end{align}
Noticing that $|\nabla^2f-\frac{1}{n}g\Delta_gf|^2_g=|\nabla^2f|^2_g - \frac{1}{n} |\Delta_gf|^2$, we get
\begin{align*}
||\nabla^2f-\frac{1}{n}g\Delta_gf||^2_{L^2}&=||\nabla^2f||^2_{L^2} - \frac{1}{n} ||\Delta_gf||^2_{L^2}\\
&=\frac{n-1}{n} ||\Delta_gf||^2_{L^2} - \int_M{\mathrm{Ric}_g}(\nabla f,\nabla f)\mu_g.
\end{align*}
Thus, notice that, parallel to what happens in the compact scenario, if $\mathrm{Ric}_g$ is non-negative, then we can get an estimate of the $L^2$-norm of $\nabla^2f-\frac{1}{n}g\Delta_gf$ in terms of the $L^2$ norm of $\Delta_gf$. Doing this, we can prove the second statement in the Lemma. That is, if $\mathrm{Ric}_g$ is non-negative, then putting the above identity together with (\ref{AS4})-(\ref{AS5}), we get
\begin{align}
||b||^2_{L^2}\leq \sqrt{n(n-1)}||\overset{\circ}B||_{L^2}||b||_{L^2}.
\end{align}

Nevertheless, we can actually get rid of this limitation on the Ricci tensor with the following argument. Notice that, from the point wise continuity of contractions, we get
\begin{align*}
\Big\vert \int_M\mathrm{Ric}_g(\nabla f,\nabla f)\mu_g \Big\vert \lesssim \int_{M}|\mathrm{Ric}_g|_g|\nabla f|^2_g\mu_g. 
\end{align*}
Also, since $\mathrm{Ric}_g\in H_{s+1,\delta+2}$, with $s>\frac{n}{2}$ and $\delta>-\frac{n}{2}$, then $H_{s+1,\delta+2}\hookrightarrow C^0_{2}$, where 
\begin{align*}
||\mathrm{Ric}_g||_{C^0_2}=\sup_{M}|\mathrm{Ric}_g|_g(1+d^2_e),
\end{align*}
thus, we get
\begin{align*}
\Big\vert \int_M\mathrm{Ric}_g(\nabla f,\nabla f)\mu_g \Big\vert &\lesssim ||\mathrm{Ric}_g||_{C^0_2}\int_M|\nabla f|^2_g(1+d^2)^{-1}\mu_g\\
&=||\mathrm{Ric}_g||_{C^0_2}\int_M|\nabla f|^2_g(1+d^2)^{\rho-1}(1+d^2)^{-\rho}\mu_g\\
&\leq ||\mathrm{Ric}_g||_{C^0_2}||\nabla f||^2_{L^2_{\rho-1}}.
\end{align*}
Notice, again, that since $f\in H_{s+3,\rho-2}$, with $\rho\geq 0$, then $\nabla f\in L^2_{\rho-1}$. Thus, the above expression is well-defined. We then get the following estimate, where $C_n$ stands for some constant depending only on $n$,
\begin{align*}
||\nabla^2f-\frac{1}{n}g\Delta_gf||^2_{L^2}&\leq \frac{n-1}{n} ||\Delta_gf||^2_{L^2} + C_{n}||\mathrm{Ric}_g||_{C^0_2}||\nabla f||^2_{L^2_{\rho-1}},\\
&\leq \frac{n-1}{n} ||\Delta_gf||^2_{L^2} + C_n||\mathrm{Ric}_g||_{C^0_2}||f||^2_{H_{2,\rho-2}}
\end{align*}
Now, notice that $\rho-2>-\frac{n}{2}$, and since $\Delta_g:H_{2,\rho-2}\mapsto H_{0,\rho}$ is injective for $\rho-2>-\frac{n}{2}$, we have the following elliptic estimate (see \cite{CB1}):
\begin{align*}
||f||_{H_{2,\rho-2}}\leq C ||\Delta_gf||_{H_{0,\rho}}\leq C ||b||_{L^2},
\end{align*}
for some constant $C$, which does not depend on the particular $f\in H_{2,\rho-2}$, where the last inequality is a consequence of the embedding $L^2_{\rho}\subset L^2$ for $\rho\geq 0$. This implies that
\begin{align*}
||\nabla^2f-\frac{1}{n}g\Delta_gf||_{L^2}&\leq \left( \frac{n-1}{n} + C||\mathrm{Ric}_g||_{C^0_2}\right)^{\frac{1}{2}}||b||_{L^2},
\end{align*}
For some other constant $C$. Putting this together with (\ref{AS4})-(\ref{AS5}), we get
\begin{align*}
||b||^2_{L^2}\leq n\left( \frac{n-1}{n} + C||\mathrm{Ric}_g||_{C^0_2}\right)^{\frac{1}{2}}||\overset{\circ}{B}||_{L^2}||b||_{L^2} , 
\end{align*}
\end{proof}

\begin{coro}
Under the same hypotheses of the previous lemma, if $n\geq 5$, then (\ref{AS.01}) can be simplified by the following estimate
\begin{align}
||b||^2_{L^2}\leq nC_g ||\overset{\circ}B||_{L^2},
\end{align}
where the constant $C_g$ is the best constant satisfying the elliptic estimate $||h||_{H_{2,-2}}\leq C_g||\Delta_gh||$ for all $h\in H_{2,-2}$.
\end{coro}
\begin{proof}
Under these assumptions we can look at $b\in H_{s+1,0}$, and consider the map $\Delta_g:H_{s+3,-2}\mapsto H_{s+1,0}$, which will give us a solution $f\in H_{s+3,-2}$ to (\ref{AS1}). Then, from (\ref{AS4})-(\ref{AS5}), we can actually deduce that
\begin{align*}
||b||^2_{L^2}\leq n ||\overset{\circ}B||_{L^2}||\nabla^2f||_{L^2}.
\end{align*}
Thus, since $||\nabla^2f||_{L^2}\leq ||\nabla^2f||_{H_{2,-2}}$, and, again, for $n\geq 5$ we have that $\Delta_g$ is injective on $H_{2,-2}$, we get the elliptic estimate  $||\nabla^2f||_{H_{2,-2}}\leq C_g||b||_{L^2}$. In this case, we would directly get
\begin{align}
||b||^2_{L^2}\leq nC_g ||\overset{\circ}B||_{L^2}||b||_{L^2},
\end{align}
which is a simpler estimate than the one presented in the main proof.
\end{proof}

From this general theorem, we get a version of the usual almost-Schur inequality for the scalar-curvature \cite{Topping}-\cite{Barbosa} in the context of AE manifolds. In our case, we are able to state such inequality without restriction on the Ricci tensor, but we do not provide an explicit value for the constant $C$. 

\begin{coro}\label{R-AS}
Let $(M^n,g_0)$ is an $H_{s+3,\delta}$-asymptotically euclidean manifold without boundary, with $n\geq 3$, $\delta>-\frac{n}{2}$ and $\delta\geq-2$ and $s>\frac{n}{2}$. Then, there is a constant $C$, such that for any $H_{s+3,\delta}$-AE metric $g$ on $M$ the following inequality holds
\begin{align}
||R_g||_{L^2}\leq C ||\overset{\circ}{\mathrm{Ric}_g}||_{L^2}
\end{align}
\end{coro}
\begin{proof}
Notice that since $g$ is $H_{s+3,\delta}$-AE, with $\delta>-\frac{n}{2}$ and $\delta\geq -2$, then we have $\mathrm{Ric}_g,R_g\in H_{s+1,\rho}$, for $\rho>-\frac{n}{2}+2$ and $\rho\geq 0$, which also implies $G_g\in H_{s+1,\rho}$ for the same $\rho$. Hence, since $\mathrm{div}_gG_g=0$, we get as an immediate corollary of the above Lemma that
\begin{align*}
||R_g||_{L^2}\leq C ||\overset{\circ}{\mathrm{G}_g}||_{L^2}.
\end{align*}
Finally, since $\overset{\circ}{\mathrm{G}_g}=\overset{\circ}{\mathrm{Ric}_g}$ our claim holds.
\end{proof}

Clearly, the following also holds as a direct consequence of the above Lemma.
\begin{coro}
Let $(M^n,g)$ is an $H_{s+3,\delta}$-asymptotically euclidean manifold without boundary, with $n\geq 3$, $\delta> - \frac{n}{2}$, $\delta\geq -2$ and $s>\frac{n}{2}$. Suppose that there is an isometric immersion $(M^n,g)\hookrightarrow (\tilde{M}^n,\tilde{g})$ into a Ricci-flat space, \textit{i.e}, $\mathrm{Ric}_{\tilde{g}}=0$. Then, if the extrinsic curvature $K\in H_{s+1,\rho}$ for $\rho>-\frac{n}{2}+2$ and $\rho\geq 0$, then, there is a constant $C$, independent of the immersion, such that
\begin{align}\label{AS-K}
||\tau||_{L^2}\leq C ||\overset{\circ}{K}||_{L^2},
\end{align} 
where $\tau\doteq\mathrm{tr}_gK$.
\end{coro}
\begin{proof}
As we have already commented, under our hypotheses, we have that the Codazzi equation for hypersurfaces guarantees that the tensor $\pi=K - \tau g\in H_{s+1,0}$ and satisfies $\mathrm{div}_g\pi=0$. Furthermore, we have that $\overset{\circ}{\pi}=\overset{\circ}{K}$. Thus (\ref{AS}) proves (\ref{AS-K}).
\end{proof}

Finally, we will present an almost-Schur-type inequality for $Q$-curvature on AE-manifolds. Following \cite{Topping}-\cite{Barbosa}, it has been shown in \cite{chinos} that such an inequality holds in the context of closed manifolds.  
It should be stress that problems related with $Q$-curvature have gained plenty of attention, see \cite{Qcurv} for a survey. Before presenting this result, we should introduce some definitions. First of all, given a Riemannian manifold $(M^n,g)$, with $n\geq 3$, we will define its $Q$-curvature in the following way
\begin{align}\label{Qcurv}
Q_g\doteq A_n\Delta_gR_g + B_n|\mathrm{Ric}_g|^2_g + C_nR^2_g,
\end{align}
where $A_n=-\frac{1}{2(n-1)}$, $B_n=-\frac{2}{(n-2)^2}$ and $C_n=\frac{n^2(n-4)+16(n-1)}{8(n-1)^2(n-2)^2}$. 

Now, following \cite{chinos}, we will explain how we can canonically associate to $Q$-curvature a divergence-free $(0,2)$-symmetric tensor field.
Notice that we can view $Q$-curvature as a map on the set of Riemannian metrics on a given manifold $M$, given by
\begin{align*}
Q:\mathrm{Riem}^{k}(M)&\mapsto C^{k-4}(M),\\
g &\mapsto Q_g,
\end{align*} 
where $\mathrm{Riem}^{k}(M)$ denotes the set of $C^k$ Riemannian metrics on $M$ and we are considering $k\geq 4$. Then, we can analyse its linearisation $L_g\doteq D_gQ:S^2M\mapsto C^{k-4}(M)$ which is given by an operator acting on the space of $C^k$-symmetric tensor fields on $M$. Let $L^{*}_g$ be its formal adjoint operator, which acts on functions on $M$ and produces a symmetric tensor-field on $M$. Now, consider any point $p\in M$ and a bounded neighbourhood $U$ of $p$. Furthermore, let $V\subset U$ be another neighbourhood of $p$. Let $f\in C^{\infty}_{0}(U)$ and $X$ be a compactly supported smooth vector field on $U$. Then, we have that
\begin{align*}
\int_U\langle X,\mathrm{div}_g(L^{*}_gf) \rangle \mu_g = -\frac{1}{2}\int_U\langle \pounds_Xg,L^{*}_gf \rangle \mu_g=-\frac{1}{2}\int_U f L_g(\pounds_Xg) \mu_g.
\end{align*}
Now, let $\psi_t$ be the flow associated to the vector field $X$, then, we have that
\begin{align*}
X(Q_g)=\frac{d}{dt}\psi^{*}_t(Q_{g})|_{t=0}=\frac{d}{dt}Q_{\psi^{*}_tg}|_{t=0}=D_{\psi^{*}_tg}Q\cdot\frac{d}{dt}\psi^{*}_tg|_{t=0}=D_{g}Q\cdot\pounds_Xg=L_g(\pounds_Xg),
\end{align*}
thus, we have that
\begin{align*}
\int_U\langle \mathrm{div}_g(L^{*}_gf),X \rangle \mu_g = -\frac{1}{2}\int_U \langle fdQ_g, X \rangle \mu_g, 
\end{align*}
which implies that $\mathrm{div}_g(L^{*}_gf)+\frac{1}{2}fdQ_g=0$ $\forall$ $f\in C^\infty_0(U)$. In particular, if we consider $f|_{V}\equiv 1$, then we get that, in a neighbourhood of $p$, it holds that
\begin{align*}
\mathrm{div}_g(L^{*}_g1)+\frac{1}{2}dQ_g=\mathrm{div}_g\big( L^{*}_g1 + \frac{1}{2}Q_gg \big)  =0.
\end{align*}
In particular, this means that if we define the tensor field $J_g\doteq -\frac{1}{2}L^{*}_g1$ then,  $B_J\doteq  J_g - \frac{1}{4}Q_gg$ is a conserved $(0,2)$-symmetric tensor field. The explicit expression for such tensor field is given in terms of the $Q$-curvature, the Bach tensor and the Schoutten tensor, and, as a map on the metric, is a fourth order operator. Explicitly, it can be written as follows \cite{chinos}.
\begin{align*}
\begin{split}
&J_g=\frac{1}{n}Q_gg-\frac{1}{n-2}B_g-\frac{n-4}{4(n-1)(n-2)}T_g,\\
&{B_g}_{jk}\doteq \nabla^i{C_g}_{ijk} + {W_g}_{ijkl}S_g^{il},\\
&{C_g}_{ijk}\doteq \nabla_i{S_g}_{jk} - \nabla_j{S_g}_{ik},\\
&S_g\doteq \frac{1}{n-2}(\mathrm{Ric}_g - \frac{1}{2(n-1)}R_gg),\\
&W_g\doteq \mathrm{Riem}_g - \frac{1}{n-2}\overset{\circ}{\mathrm{Ric}}\owedge g - \frac{R_g}{2n(n-1)}g\owedge g,\\
&h\owedge k(X,Y,Z,V)\doteq h(X,Z)k(Y,W) + h(Y,V)k(X,Z) - h(X,V)k(Y,Z) - h(Y,Z)k(X,V),\\
&T_g\doteq (n-2)(\nabla^2\mathrm{tr}_gS_g - \frac{g}{n}\Delta_g\mathrm{tr}_gS_g) + 4(n-1)(S\times S - \frac{1}{n}|S_g|^2_g g) - n\mathrm{tr}_g\mathrm{S}_g \overset{\circ}{S}_g,\\
&(S\times S)_{jk}\doteq S^i_jS_{ik},
\end{split}
\end{align*}
where in the above we have defined several useful tensor and operations, namely, the Bach tensor $B_g$; the Cotton tensor $C_g$; the Schouten tensor $S_g$; the Weyl curvature tensor $W_g$, which uses the $\owedge$-operator, which for two $(0,2)$-tensor fields $h$ and $k$ produces a $(0,4)$-tensor field. Since the Bach tensor is known to be traceless, and from the above expressions we have that $\mbox{tr}_gT_g=0$, then $\mathrm{tr}_g J_g=Q_g$. Taking into consideration all this notation, we are now in a position to present the following result, whose proof runs in complete analogy to the previous ones.

\begin{coro}
Let $(M^n,g_0)$ is an $H_{s+3,\delta_0}$-asymptotically euclidean manifold without boundary, with $n\geq 3$, $\delta>-\frac{n}{2}$, $\delta_0\geq-2$ and $s>\frac{n}{2}$. Then, there is a constant $C$, such that for any $H_{s+5,\delta}$-AE metric $g$, with $\delta>-\min\{\frac{n}{2}+2,4\}$, the following inequality holds
\begin{align}
||Q_g||_{L^2}\leq C ||\overset{\circ}{J_g}||_{L^2}
\end{align}
\end{coro}
\begin{proof}
We basically only need to show that $J_g$ and $Q_g$ satisfy the hypotheses of  Lemma \ref{B-AS}. Since $g$ is $H_{s+5,\delta}$-AE with $\delta>-\frac{n}{2}$, then from the continuous multiplication property and the fact that  $\mathrm{Riem}_g, \mathrm{Ric}_g\in H_{s+3,\delta+2}$, we have that $W_g,S_g\in H_{s+3,\delta+2}$. Thus, since contractions define continuous operations on tensor fields, the multiplication property gives us that the contracted tensor ${W_g}_{\cdot}S\doteq {W_g}_{ijkl}S^{il}$ is in $H_{s+3,\delta+4}$. Also, we get that $C_g\in H_{s+2,\delta+3}$, which, together with ${W_g}_{\cdot}S\in H_{s+3,\delta+4}$, gives us $B_g\in H_{s+1,\delta+4}$. Similarly, we have that $\Delta_gR_g\in H_{s+1,\delta+4}$ and $R^2_g,|\mathrm{Ric}_g|^2_g\in H_{s+3,\delta+4}$, which gives us that $Q_g\in H_{s+1,\delta+4}$. Finally, the fact that $S\in H_{s+3,\delta+2}$ and the multiplication property give us that $T\in H_{s+1,\delta+4}$. All this implies that $Q_g,J_g\in H_{s+1,\delta+4}$, where $\rho=\delta+4\geq 0$ and $\rho>-\frac{n}{2}+2$, since $\delta>-\min\{\frac{n}{2}+2,4\}$. Thus, we are under the hypotheses of Lemma \ref{B-AS}, and the result follows.
\end{proof}
Finally, it is worth noticing that Lemma \ref{B-AS} has been stated in a quite general form, that, as shown above, can be adapted to analyse different versions of almost-Schur type inequalities for different interesting geometric tensors related to conserved quantities. Obviously, we cannot exhaust all the examples here, but, surely, there are other rather obvious candidates for such examples, such us Lovelock curvatures, which have been in the center of plenty of resent research in geometric analysis.

\subsection*{Static potentials}

Recall that a Riemannian metric $g$ on a manifold $M$ is called static if the linearised scalar curvature map has nontrivial cokernel. This is equivalent to stating that the equation
\begin{align}
-g(\Delta_gf)+\nabla^2f-fRic_g=0
\end{align}
admits some nontrivial solution $f$. Such solution is referred to as a \textit{static potential}. It has been shown in \cite{Corvino} that any $C^3$-static metric must have constant scalar curvature, which implies that any $C^3$-asymptotically-flat static metrics must have zero scalar curvature. This means that a static potential of an asymptotically flat static metric must satisfy the system\footnote{Notice that this means that no static potential can exist in $H_{s,\delta}$ for $\delta>-\frac{n}{2}$, since $\Delta_g$ is injective in this case.}
\begin{align}\label{staticeq}
\begin{split}
\nabla^2f&=fRic_g,\\
\Delta_gf&=0.
\end{split}
\end{align}

Positive static potentials are an interesting object of study, since they are intimately related with vacuum space-time solutions of the Einstein equations, as was shown in \cite{Corvino}. In fact, there it is shown that a Riemannian manifold $(M^n,g)$ admits a static potential iff the warped product metric $\tilde{g}=-f^2dt\otimes dt + g$ is Einstein. Thus, in the case of $(M,g)$ being an AE manifold, since $R(g)=0$, then $\tilde{g}$ must actually be Ricci-flat. Thus, we get a correspondence between \textit{static} vacuum solutions of the space-time Einstein equations with static Riemannian manifolds. Furthermore, in \cite{Corvino} it is also shown that the zero level set of a static potential (in case it exists), is given by a regular totally geodesic hypersurface in $M$. Notice that in case $f^{-1}(0)\neq \emptyset$, then, the space-time metric $\tilde{g}$ will actually degenerate on such set, possibly signalling the existence of some \textit{pathology} of the space-time structure. In fact, since the scalar curvature of a static Riemannian metric is zero, then it is trivially a solution of the \textit{time-symmetric} vacuum Einstein constraint equations, which under evolution generate the static space-time. In this scenario, the hypersurface $f^{-1}(0)$ defines an \textit{apparent horizon} in the initial data. Such structures are related with the formation of black holes under the space-time evolution. For some details concerning these ideas, see, for instance, \cite{maxwell}-\cite{Corvino}-\cite{chrusciel-mazzeo} and references therein.  

Taking into account the above ideas, it is not surprissing that static potentials in the context of AE manifolds have atractted quite a lot of attention, and some strong rigidity has been observed for instance in \cite{Miao} for the $3$-dimensional case. In this section, we intend to show how some higher-dimensional analogues of their results follow naturally from the Pohozaev-Schoen identity (\ref{PS}), plus some mild hypotheses on $f$. Before going to the main statements, we make explicit the fact that throughout this section, when we consider static AE manifolds with boundary, we will consider that $\partial M=f^{-1}(0)$, this being motivated by the above discussion. That this, we have a model in mind where the boundary components would signal the existence of apparent horizons in time-symmetric vacuum initial data sets for the Einstein equations, which, under evolution, should evolve into black hole static space-times. The following lemma, which can be derived using some results that can be found \cite{Miao}, is presented within our current notations and conventions. We include the proof for sake of completeness.

\begin{lemma}
Suppose that $f$ is a static potential of an $H_{s+3,\delta}$-asymptotically euclidean metric $g$, with $s>\frac{n}{2}$, $\delta>-1$, and $n\geq 3$. Then $f\in C^2$ has bounded gradient $\nabla f$ and $\nabla^2f\in L^2$.
\end{lemma}
\begin{proof}
First of all, concerning the regularity of the solution, since $H_{s+3,\delta}$ embeds in $C^3$, appealing to Proposition 2.5 in \cite{Corvino}, we get that the static potential is $C^2$. Now, consider a point $x\in M$ which lies in one of the ends $E_i$. Let $B_{r_0}(p)$ be the $g$-geodesic ball of radius $r_0$ around the origin $p$, with $r_0$ chosen large enough such that its boundary $\partial B_{r_0}$ lies in the ends of $M$, and, furthermore, suppose that $x$ has been chosen so that it lies outside this geodesic ball. Let $\gamma(t)$ be a $g$-minimizing geodesic joining $p$ and $x$, parametrized by arc-length, so that $\gamma(r_0)\in\partial B_{r_0}$, and suppose that $\gamma(T)=x$. Finally, consider $f(t)\doteq f(\gamma(t))$, $r_0\leq t\leq T$. Then, since $f$ is a static potential of $g$, we have
\begin{align}\label{miao1}
f''(t)=\mathrm{Ric}_g(\gamma',\gamma')f(t) \;\;\ \forall \;\; r_0\leq t\leq T
\end{align}
Since $g-e\in H_{s+2,\delta}$, then $\mathrm{Ric}_g\in H_{s,\delta+2}\subset L^2_{\delta+2}\cap C^0$, which implies that $|\mathrm{Ric}_g|_e=o(d_e(x)^{-(\delta+2+\frac{n}{2})})$ near infinity, in each end. The asymptotic condition on $g$ actually gives us that $|\mathrm{Ric}_g|_e=o(d_g(x)^{-(\delta+2+\frac{n}{2})})$. This implies that, given an arbitrary $\epsilon>0$, if $t$ is sufficiently large, then 
\begin{align}\label{miao2}
|\mathrm{Ric}_g(\gamma'_t,\gamma'_t)|(\gamma_t)\leq \epsilon d_g(\gamma_t)^{-(\delta+\frac{n}{2}+2)}\leq \epsilon t^{-2}.
\end{align}

Now, define $\alpha\doteq \frac{1}{2}(1+\sqrt{1+4\epsilon})$; $a\doteq \sup_{\partial B_{r_0}}(|f|+|Df|_e)$ and $\omega(t)\doteq At^{\alpha}$, with $r_0\leq t\leq T$, where $A$ is a constant chosen such that $Ar_0^{\alpha}>a$ and $\alpha Ar_0^{\alpha-1}>a$ and $\epsilon$ is some positive constant. Then, the function $\omega(t)$ satisfies the following properties
\begin{align}\label{miao3}
\omega''(t)=\epsilon t^{-2}\omega(t), \;\; |f(r_0)|\leq \omega(r_0), \;\; |f'(r_0)|\leq \omega'(r_0).
\end{align}

Suppose that $|f(t)|>\omega(t)$ for some $r_0\leq t\leq T$. Then, define $t_1\doteq \inf \{t\in [r_0,T] \;\; \backslash \;\;  |f(t)|>\omega(t)  \}$. We know that $t_1$ satisfies that $t_1>r_0$ and $f(t_1)=\omega(t_1)$. Using (\ref{miao1})-(\ref{miao2}), we have that 
\begin{align*}
|f''(t)|\leq \epsilon t^{-2}\omega(t)=\omega''(t) \;\; \forall \;\; r_0\leq t\leq t_1.
\end{align*}
Now, integrating he above inequality twice and using (\ref{miao3}), gives $|f(t)|\leq \omega$ for all $r_0\leq t\leq t_1$, which is a contradiction, thus we get that $|f(t)|\leq \omega(t)$ for all $t\in [r_0,T]$. Thus, since $\delta>-\frac{n}{2}$, we can choose an $\epsilon>0$ such that $\alpha<1+\frac{1}{2}(\delta+\frac{n}{2})$, which implies that
\begin{align*}
|f(t)|\leq At^{1+\frac{1}{2}(\delta+\frac{n}{2})}.
\end{align*}
The above inequality together with (\ref{miao1})-(\ref{miao2}), gives that
\begin{align*}
|f''(t)|\leq A\epsilon t^{-(\delta+2+\frac{n}{2})}t^{1+\frac{1}{2}(\delta+\frac{n}{2})}= A\epsilon t^{-(1+\frac{1}{2}(\delta+\frac{n}{2}))} \;\; \forall \;\; r_0\leq t \leq T.
\end{align*}
This implies that there is a constant $C_1$, independent of the point $x\in E_i\backslash B_{r_0}$ and of $t$, such that 
\begin{align}\label{miao4}
|f''(t)|\leq C_1 t^{-(1+\frac{1}{2}(\delta+\frac{n}{2}))} \;\; \forall \;\; r_0\leq t \leq T.
\end{align}
Integrating this inequality between $r_0$ and $t$, we can show that $|f'(t)|\leq C_2$ for some other constant $C_2$, which shows that for sufficiently large $|x|$ it holds that
\begin{align}
|f|(x)\leq C_3|x|.
\end{align}

Now, using the equation $\nabla^2f=f\mathrm{Ric}_g$, we have that $|\nabla^2f|_e(x)\lesssim |x||x|^{-(\delta+2+\frac{n}{2})}=|x|^{-(1+\delta+\frac{n}{2})}\in L^2(\mathbb{R}^n\backslash B_1(0))$ for $\delta>-1$, which proves that $\nabla^2f\in L^2$. In order to prove that $\nabla f$ is bounded, define $\phi\doteq|\nabla f|^2_g$, and notice that 
\begin{align*}
|\nabla\phi|^2_g(x)\leq 4|\nabla^2f|^2_g\phi 
\end{align*}
Doing as above, consider $\phi(t)\doteq \phi(\gamma(t))$, and notice that since $g$ is AE, then $g$ and $e$ define \textit{equivalent} metrics, that is, for any tangent vector $v_p$ to $M$, there are positive constants $a$ and $b$, independent of the point $p\in M$, such that $a|v_p|_e\leq |v_p|_g \leq b|v|_e$. Thus, the above inequality tells us that $|\nabla\phi|^2_e(x)\lesssim |\nabla^2f|^2_e\phi$. Then, we have that
\begin{align*}
|\phi'(t)|=|g(\nabla\phi_{\gamma(t)},\gamma'_t)|\leq |\nabla\phi|_g\leq 2|\nabla^2f|_g\phi^{\frac{1}{2}} \lesssim |\nabla^2f|_e\phi^{\frac{1}{2}} \lesssim t^{-(1+\delta+\frac{n}{2})}\phi^{\frac{1}{2}}(t),
\end{align*}
which, after integration, gives us that
\begin{align*}
\frac{1}{\delta+\frac{n}{2}}\left( t^{-(\delta+\frac{n}{2})} - r^{-(\delta+\frac{n}{2})}_0\right)\lesssim \phi^{\frac{1}{2}}(t) -\phi^{\frac{1}{2}}(r_0) \lesssim -\frac{1}{\delta+\frac{n}{2}}\left( t^{-(\delta+\frac{n}{2})} - r^{-(\delta+\frac{n}{2})}_0\right).
\end{align*}
Since $t> r_0$ and $\delta+\frac{n}{2}>0$, the above inequality implies that
\begin{align*}
\phi^{\frac{1}{2}}(r_0)-\frac{1}{\delta+\frac{n}{2}} r^{-(\delta+\frac{n}{2})}_0\lesssim \phi^{\frac{1}{2}}(t) \lesssim \phi^{\frac{1}{2}}(r_0)+ \frac{1}{\delta+\frac{n}{2}} r^{-(\delta+\frac{n}{2})}_0 \;\;\; \forall \;\;\; t> r_0.
\end{align*}
The above relation implies that $|\nabla f|_g(x)=\phi^{\frac{1}{2}}(x)$ is bounded.

\end{proof}

\begin{thm}
Suppose that $(M^n,g)$ is a $H_{s+3,\delta}$-AE manifold with $n\geq 3$, $s>\frac{n}{2}$ and $\delta>-1$, which admits a non-negative static potential $f$. Furthermore, suppose that $\mathrm{Ric}_g\in H_{s+1,\rho}$, for some $\rho>\frac{n}{2}-1$, and that $\partial M=f^{-1}(0)$ consists of $N$ closed connected components, labelled by $\{\Sigma_i \}_{i=1}^{N}$. Then, it follows that 
\begin{align}
\int_Mf|\mathrm{Ric}_g|^2_g\mu_g = \frac{1}{2}\sum_{i=i}^Nc_i\int_{\Sigma_i}R_{h_i}\mu_{\partial M},
\end{align}
where the constants $c_i=|\nabla f|_{\Sigma_i}$. In particular, if $\partial M=\emptyset$, then $(M^n,g)$ is isometric to $(\mathbb{R}^n,e)$, where $e$ is the euclidean metric.
\end{thm}
\begin{proof}
First, notice that under our hypotheses the above lemma shows that $X=\nabla f$ satisfies the general hypotheses of Theorem \ref{thm1} if we pick $\rho>\frac{n}{2}-1$. Also, since a static AE metric has zero scalar curvature, we get that $\mathrm{Ric}_g\in H_{s+1,\rho}$ is a conserved $(0,2)$-tensor field. Thus, our choice of $\rho>\frac{n}{2}-1$ shows that we are under the hypotheses of Theorem \ref{thm1}. Then, we can apply (\ref{PS}) choosing the vector field $X=\nabla f$ so as to get
\begin{align}\label{SP1}
\int_M\langle \mathrm{Ric}_g,\nabla^2f \rangle\mu_g=\int_{\partial M}\mathrm{Ric}(\nabla f,\nu)\mu_{\partial M}.
\end{align}
Also, from the static equation, we have that $\langle \mathrm{Ric}_g,\nabla^2f \rangle=f|\mathrm{Ric}_g|^2_g$. Furthermore, since $\partial M=f^{-1}(0)$, then, we get that $\nu=-\frac{\nabla f}{|\nabla f|}$. In addition, the static equation implies that on $\partial M$ it holds that $\nabla^2f(X,\nabla f)=0$ for any $X$ tangent to $\partial M$. Then, since $\nabla^2f(X,\nabla f)=\frac{1}{2}X(|\nabla f|^2_g)$, we get that $|\nabla f|$ is constant along each connected component of $\partial M$. Thus, we see that
\begin{align*}
\int_Mf|\mathrm{Ric}_g|^2_g\mu_g = -\sum_{i=i}^Nc_i\int_{\Sigma_i}\mathrm{Ric}(\nu,\nu)\mu_{\partial M},
\end{align*}
where the sum is carried along the $N$ connected components $\{\Sigma_i\}_{i=1}^N$ of $\partial M$, and the constants $c_i=|\nabla f|_{\Sigma_i}$. Now, denote by $h_i$ the induced Riemannian metric on $\Sigma_i$. Now, let $\{E_i,\nu\}_{i=1}^{n-1}$ denote an orthonormal frame of $M$ along $\Sigma_i$, where $\{ E_i \}_{i=1}^{n-1}$ gives an orthonormal frame on $\Sigma_i$. Then, we get that for any pair of tangent vectors $X,Y$ to $\Sigma_i$ it holds that
\begin{align*}
\mathrm{Ric}_g(X,Y)&=\sum_{i=1}^{n-1}g(R_g(X,E_i)E_i,Y) + g(R_g(X,\nu)\nu,Y),
\end{align*}
which, from the Gauss equation and the fact that all the components are totally geodesic, implies that
\begin{align*}
R_g&=\sum_{j=1}^{n-1}\mathrm{Ric}_g(E_j,E_j) + \mathrm{Ric}_g(\nu,\nu)=\sum_{j=1}^{n-1}\sum_{i=1}^{n-1}g(R_g(E_j,E_i)E_i,E_j)  + \sum_{j=1}^{n-1}g(R_g(E_j,\nu)\nu,E_j) \\
&+ \mathrm{Ric}_g(\nu,\nu),\\
&=\sum_{j=1}^{n-1}\sum_{i=1}^{n-1}h(R_{h_i}(E_j,E_i)E_i,E_j) + \sum_{j=1}^{n-1}g(R_g(\nu,E_j)E_j,\nu) + \mathrm{Ric}_g(\nu,\nu),\\
&=R_{h_i} + 2\mathrm{Ric}_g(\nu,\nu).
\end{align*}
Again, since $R_g=0$, we get $-\mathrm{Ric}_g(\nu,\nu)|_{\Sigma_i}=\frac{1}{2}R_{h_i}$. Thus, going back, this gives us that
\begin{align}\label{static1}
\int_Mf|\mathrm{Ric}_g|^2_g\mu_g = \frac{1}{2}\sum_{i=i}^Nc_i\int_{\Sigma_i}R_{h_i}\mu_{\partial M}.
\end{align}

From the above equation it is clear that if $\partial M=\emptyset$, then $\mathrm{Ric}_g=0$ on $M$. In such case, we have already shown this implies that $(M,g)$ has to be isometric to $(\mathbb{R}^n,e)$.
 
\end{proof}

\begin{remark}
Notice that the two conditions $g-e\in H_{s+3,\delta}$, with $\delta>-1$, and $\mathrm{Ric}_g\in H_{s+1,\rho}$, with $\rho>\frac{n}{2}-1$, in dimensions 3 and 4 are actually redundant, since, from $g-e\in H_{s+3,\delta}$, we get that $\mathrm{Ric}_g\in H_{s+1,\delta+2}$, with $\delta+2>1$. Thus, in dimension $n=3$, we get $\frac{n}{2}-1=\frac{1}{2}<\delta+2$. Similarly, for $n=4$, we get $\frac{n}{2}-1=1<\delta+2$. Nevertheless, for $n\geq 5$ the condition on the Ricci tensor becomes a necessary additional information. 

From the above discussion, we see that the identity (\ref{static1}) is presented for $n=3$ in \cite{Miao}, and that the same identity holds with the same hypotheses for $n=4$. Nevertheless, for $n\geq 5$ the identity extends naturally, although non-trivially. 
\end{remark}

\bigskip
Now, consider a connected static manifold $(M^n,g,f)$, with $n\geq 3$, where $g$ and $\mathrm{Ric}_g$ satisfy the hypotheses of the above Lemma, and suppose that $\partial M=\emptyset$. Notice that if $f$ changes sign, then $f^{-1}(0)\neq \emptyset$ actually represents the boundary of both $f^{-1}(0,\infty)$ and $f^{-1}(-\infty,0)$, which are two $n$-dimensional submanifolds of $M$. We can then add $f^{-1}(0)$ to any of these submanifolds, generating two manifolds with boundary, say $M_{+}$ and $M_{-}$ respectively. Clearly, $(M_{+},g,f)$ is a static manifold with boundary which satisfies all the hypotheses of the above theorem provided suitable decaying conditions for $g$. Furthermore, $(M_{-},g,-f)$ is also a static manifold since the static equation is linear. Thus, $(M_{-},g,-f)$ also satisfies all the hypotheses of the above theorem. Thus, we get the following corollary.

\begin{coro}
Let $(M^n,g,f)$ be a connected static manifold, with $n\geq 3$, where $g$ and $\mathrm{Ric}_g$ satisfy the hypotheses of the above theorem. Moreover, suppose that $f^{-1}(0)\neq \emptyset$ and $\partial M=\emptyset$. Then, we have that
\begin{align}\label{balckholetopology.1}
\sum_{i=i}^Nc_i\int_{\Gamma_i}R_{h_i}\mu_{\partial M} \geq 0,
\end{align}
where $\{\Gamma_i\}_{i=1}^N$ denote the connected components of $f^{-1}(0)$. Furthermore, the equality holds if and only if $(M,g)$ is isometric to $(\mathbb{R}^n,e)$.
\end{coro}
\begin{proof}
Consider the notation introduced above. Then, under our hypotheses, we have that both $(M_{+},g,f)$ and $(M_{-},g,-f)$ are static manifolds with boundary, for which the above theorem holds. Then, we get that
\begin{align*}
\sum_{i=i}^Nc_i\int_{\Gamma_i}R_{h_i}\mu_{\partial M}=2\int_{M_{+}}f|\mathrm{Ric}_g|^2_g\mu_g\geq 0,\\
\sum_{i=i}^Nc_i\int_{\Gamma_i}R_{h_i}\mu_{\partial M}=-2\int_{M_{-}}f|\mathrm{Ric}_g|^2_g\mu_g \geq 0,
\end{align*}
where, in both cases, the equality holds iff $\mathrm{Ric}_g=0$ on both $M_{+}$ and $M_{-}$. That is, the equality holds iff $\mathrm{Ric}_g=0$ on $M$. But then, from Theorem \ref{Ricciflatrigidity}, we get the rigidity statement in the equality case.
\end{proof}

From the above corollary, we get the following.

\begin{coro}
Under the same hypotheses of the above corollary, if $n=3$ and $f^{-1}(0)$ consists of one connected component, that is $N=1$, then $f^{-1}(0)$ is homeomorphic to $\mathbb{S}^2$. If $N\geq 2$, then there must be at least one topological sphere within the components of $f^{-1}(0)$.
\end{coro}

\begin{proof}
This is straightforward, since, from (\ref{balckholetopology.1}), we get that if $N=1$ the total scalar curvature of $\partial M$ must be positive. Thus, from the Gauss-Bonnet theorem, we know that the Euler characteristic of $\partial M$ must be positive, thus $\partial M$ is a topological $2$-sphere.

On the other hand, in the case $N\geq 2$, since the sum of the total scalar curvatures of the different components has to be positive, then we must have at least one sphere from the argument given above.
\end{proof}

\begin{remark}
It is important to stress that the above corollary is very closely related to well-known results related with the classification of allowable black hole topologies. Recall that the zero level set of a static potential models the intersection of the event horizon of a static black hole with a $t$-constant hypersurface in space-time. Thus, the characterization given above shows that such slices of the event horizon of an isolated static black hole have to be a topological 2-spheres. Once this connection between the two problems is clear, it also becomes clear that the above corollary is not a new result, but something that can be extracted from more general statements, such as Hawking's black hole uniqueness theorem \cite{Hawking}, or even more generally, a recent result by Eichmair, Galloway and Pollack \cite{EGP}. In fact, these results are much stronger than the above corollary, since the show that every component of $\partial M$ should be a topological $2$-sphere and they are not restricted to static space-times. Nevertheless, we consider that the value of the above corollary relies on its simplicity, since this result is a straightforward application of the generalized Pohozaev-Schoen identity, which, in turn, relies only on a conservation identity derived from a symmetry principle. In contrast, the sharper characterizations presented in \cite{Hawking} and \cite{EGP} rely on more delicate techniques. In particular, for instance, the result presented in \cite{EGP} relies on the positive resolution of the geometrization conjecture. 	
\end{remark}

\section*{Acknowledgements}

The first author would like to thank professor Paul Laurain for reading a previous version of this paper and making several relevant comments and suggestions.
 
The first author would also like to thank CAPES/PNPD for financial support.

\end{document}